\documentclass[12pt, reqno]{amsart}
\usepackage{amscd}        
\usepackage{etex}
\usepackage{enumitem}
\usepackage{mathtools, amssymb}
\usepackage{caption}
\usepackage{subcaption}
\usepackage{float, url}
\usepackage{longtable}
\usepackage{lipsum}
\usepackage{nomencl}
\usepackage{setspace}
\usepackage{algorithm}
\usepackage{algpseudocode}
\usepackage{pifont}
\usepackage{mathrsfs}
\usepackage{stmaryrd}
\usepackage{bm}
\usepackage{tikz-cd}
\usepackage{verbatim}
\usepackage{xcolor}
\usepackage{colortbl}
\usepackage{booktabs}
\usepackage{tabu}
\usepackage{amsmath}

\usepackage{color}
\usepackage{graphicx}
\usepackage{rotating}
\usepackage{diagbox}
\usepackage{amssymb}
\usepackage{epstopdf}
\usepackage{tikz}
\definecolor{mintgreen}{RGB}{152,255,152}
\definecolor{pinksalmon}{RGB}{255,102,102}
\definecolor{hueso}{RGB}{245,245,220}
\definecolor{marfil}{RGB}{255,253,208}
\definecolor{amarillo}{RGB}{255,255,0}
\usetikzlibrary{decorations.markings,arrows}
\usetikzlibrary{decorations.pathreplacing}
\usepackage[inner=1.0in,outer=1.0in,bottom=1.0in, top=1.0in]{geometry}

\usepackage{tikz-3dplot}

\newcommand{\vecq}{\vec{\mathbf{q}}}
\newcommand{\veci}{\vec{\textbf{\i}}}
\newcommand{\ord}{\textnormal{ord}}
\newcommand{\Ker}{\textnormal{Ker}}

\newcommand{\Stab}{\textnormal{Stab}}
\newcommand{\Nm}{\textnormal{Nm}}
\newcommand{\Cl}{\textnormal{Cl}}

\numberwithin{equation}{section}

\newtheorem{theorem}{Theorem}[section]
\newtheorem{lemma}[theorem]{Lemma}
\newtheorem{proposition}[theorem]{Proposition}

\makeatletter
\def\moverlay{\mathpalette\mov@rlay}
\def\mov@rlay#1#2{\leavevmode\vtop{%
   \baselineskip\z@skip \lineskiplimit-\maxdimen
   \ialign{\hfil$\m@th#1##$\hfil\cr#2\crcr}}}
\newcommand{\charfusion}[3][\mathord]{
    #1{\ifx#1\mathop\vphantom{#2}\fi
        \mathpalette\mov@rlay{#2\cr#3}
      }
    \ifx#1\mathop\expandafter\displaylimits\fi}
\makeatother

\newcommand{\suchthat}{\;\ifnum\currentgrouptype=16 \middle\fi|\;}

\newcommand{\Z}{\mathbb{Z}}

\newcommand{\Q}{\mathbb{Q}}
\newcommand{\R}{\mathbb{R}}

\newcommand{\Gal}[1]{\operatorname{Gal}#1}

\newcommand{\bb}[1]{\mathbb{#1}}

\newcommand{\ind}{\mathrm{ind}}
\theoremstyle{definition}
\newtheorem{remark}[theorem]{Remark}

\newtheorem{hypothesis}[theorem]{Hypothesis}

\def\R{{\mathbb R}}
\def\Z{{\mathbb Z}}
\def\Q{{\mathbb Q}}

\def\O_K{{\Cal{O}_{K}}}
\def\O_F{{\Cal{O}_{F}}}
\def\N_F{{\Cal{N}_{F/\Q}}}

\newcommand\Disc{\textnormal{Disc}}

\begin{document}

\title{Malle's conjecture for $G \times A$, with $G=S_3, S_4, S_5$}

\author{Riad Masri}

\address{Department of Mathematics, Texas A\&M University, Mailstop 3368, College Station, TX 77843 }

\email{masri@math.tamu.edu}

\author{Frank Thorne}

\address{Department of Mathematics, University of South Carolina, 1523 Greene Street, Columbia, SC 29208}

\email{thorne@math.sc.edu}

\author{Wei-Lun Tsai}

\address{Department of Mathematics, University of South Carolina, 1523 Greene Street, Columbia, SC 29208}

\email{weilun@mailbox.sc.edu}

\author{Jiuya Wang}

\address{Department of Mathematics, University of Georgia, Boyd Research and Education Center, Athens, GA 30602}

\email{jiuya.wang@uga.edu}

\maketitle

\begin{abstract} We prove Malle's conjecture for $G \times A$, with $G=S_3, S_4, S_5$ and $A$ an abelian group. 
This builds upon work of the fourth author, who proved this result with restrictions on the primes dividing $A$.
\end{abstract}

\section{Introduction and main results}


\subsection{Main results} In \cite{Mall1,Mall2}, Malle gave a precise conjecture for the asymptotic distribution of 
number fields of fixed degree and bounded discriminant with prescribed Galois group. In this paper, 
we will prove Malle's conjecture for number fields with Galois group $G \times A$, with $G=S_3, S_4, S_5$, and $A$ an abelian group.

The result, as well as its proof, builds upon and improve the fourth author's previous work \cite{W21}, who proved this result with a restriction 
on which primes could divide $A$.

Let $G\subset S_n$ be a degree $n$ transitive permutation group. We define the counting function 
	$$N(X, G): = \#\{ E: [E:\mathbb{Q}] = n, \Gal(E^c/\Q)\simeq G\subset S_n, \Disc(E)< X \},$$
where $\Gal(E^c/\Q)$ is not only an abstract group, but also has a permutation action on $n$ different embeddings $\sigma_i: E\rightarrow \bar{\Q}$ of the number field $E$, and $\Gal(E^c/\Q)\simeq G$ are isomorphic as permutation groups. 

For any permutation $g\in S_n$, we define $\ind(g):= n-\# \{ \text{cycles of } g \}$. Then Malle's conjecture asserts there is a constant $c(G) > 0$ such that 
\begin{align}\label{mc1}
N(X,G)  \sim c(G) X^{1/a(G)}(\log X)^{b(G)-1}
\end{align}
as $X \rightarrow \infty$ with
\begin{align*}
a(G)&:=\min \{\textrm{ind}(g):~ 1 \neq g \in G\},\\
b(G)&:= \#( \{[g] \in \textrm{Conj}(G):~\textrm{ind}(g)=a(G)\}/G_\Q).
\end{align*} 
where $G_{\Q}$ acts on the set of conjugacy classes $\textrm{Conj}(G)$ via a cyclotomic character. See \cite{Mall2} for more details. 

In this paper, we focus on $G = S_d\times A\subset S_{d|A|}$ where $G$ is a direct product of the natural permutation group $S_d\subset S_d$ and the regular permutation representation $A\subset S_{|A|}$. We prove the following:
\begin{theorem}\label{mctheorem1} Malle's conjecture (\ref{mc1}) is true for $G \times A$,
with $G = S_3, S_4, S_5$. In particular, Malle's conjecture is true for the dihedral group 
$D_6 \cong S_3 \times C_2$.
\end{theorem}

In \cite[Section 2.3]{W21} it is shown that $a(S_d \times A)=1/|A|$ and $b(S_d \times A)=1$ for all $d \geq 3$. Hence 
Theorem \ref{mctheorem1} implies that for $d=3,4,5$, 
\begin{equation}\label{eq:malle_predict}
N(X,S_d \times A)  \sim c(S_d \times A)X^{1/|A|}
\end{equation}
as $X \rightarrow \infty$.

As we explain later, 
the result continues to hold if an arbitrary finite set of splitting conditions are imposed upon the $S_d \times A$-fields being counted.

\subsection{Connection to the Colmez conjecture}
One motivation for this project arose from a connection to a conjecture of Colmez \cite{Col93} in arithmetic geometry, which we now
explain. 

A {\itshape CM field}
is a totally imaginary quadratic extension of a totally real number field.
Let $F$ be a totally real field with $\Gal(F^c/\mathbb{Q}) = G$.
If $E=FK$ is the compositum of $F$ and an imaginary quadratic field $K$ (so in particular, $\Gal(E^c/\Q) = G \times C_2)$, 
then we call $E$ a \textit{$G$-unitary CM field}.

Our interest 
in the distribution of $G$-unitary CM fields stems from our effort to understand the Colmez conjecture from an 
arithmetic statistical point of view. 

Let $\mathfrak{X}$ be an abelian variety with complex multiplication by the ring of integers of $E$.
The Colmez conjecture predicts a (deep) relationship between the Faltings height of $\mathfrak{X}$ and logarithmic derivatives 
of Artin $L$--functions at $s=0$ (this generalizes the classical Chowla-Selberg formula). 

The Galois group $\Gal(E^c/\Q)$ embeds as a subgroup of the wreath product $G \wr C_2$ (see \cite{BMT}). If $\Gal(E^c/\Q)=G \wr C_2$ 
then we call $E$ a \textit{$G$-Weyl CM field}, while if $\Gal(E^c/\Q)=G \times C_2$ then (as discussed above) we call $E$ a $G$-unitary CM field. So, 
in the former case the Galois group is as large as possible, while in the latter case the Galois group is as small as possible.

In \cite{BMT}, Barquero-Sanchez and the first and second-named authors proved that: (1) the Colmez conjecture is true for $G$-Weyl 
CM fields; and (2) the $G$-Weyl CM fields have density one in the set of all CM fields whose maximal totally real subfield has 
Galois group $G$. In other words, the Colmez conjecture is true for a \textit{random} CM field. 

The result (2) implies that the $G$-unitary CM fields have density zero. It is of great interest (and difficulty) to understand the asymptotic distribution 
of this thin set. For example, Yang and Yin \cite{YY16} proved that the Colmez conjecture is true for $S_d$-unitary CM fields, and
Theorem \ref{mctheorem1} and Remark \ref{rk:splitting} imply that for $d=3,4,5$, the number of 
$S_d$-unitary CM fields is
\begin{align}\label{eq:claim_colmez}
N_{2d}^{\textrm{unitary}}(X,S_d \times C_2)  \sim c^{\prime}(S_d \times C_2)X^{1/2}
\end{align}
as $X \rightarrow \infty$. So, not only are there infinitely many such fields, but we have a precise asymptotic formula 
for their distribution.  

Note that Yang and Yin \cite{YY16} also proved the Colmez conjecture for $A_d$-unitary CM fields, and Parenti \cite{P18} proved the Colmez 
conjecture for $\textrm{PSL}_2(\mathbb{F}_q)$-unitary CM fields. The framework of this paper can be used to give (conjectural) asymptotics for 
$N_{2d}^{\textrm{unitary}}(X,G \times A)$ with $G=A_d, \textrm{PSL}_2(\mathbb{F}_q)$.

\subsection{The distribution of $S_d$-and $A$-fields}
We now describe the relevant facts about $S_d$- and $A$-fields that will be needed in the proof. We fix an integer $d \geq 3$ and an abelian group $A$, and
define
\begin{align*}
\mathcal{F}_{S_d} &:= \{F : [F:\mathbb{Q}] = d, \ \Gal(F^c/\bb{Q}) \cong S_d\},\\
\mathcal{F}_{A} &:= \{K : \ [K:\mathbb{Q}] = |A|, \ \Gal(K/\Q) \cong A\}.
\end{align*}

We will need to count fields satisfying finite sets of {\itshape splitting conditions}, 
by which we mean the following.
For a finite set of primes $S$ and a positive integer $n$, let $\Sigma_{S, n}$ consist of a choice of \'etale $\Q_p$-algebra 
of dimension $n$ for each prime $p \in S$. 
If $E$ is a number field of degree $n$,
we say that $E \in \Sigma_{S,n}$ if for each prime $p \in S$, the \'etale $\Q_p$-algebra 
$E \otimes_{\mathbb{Q}}\mathbb{Q}_p \in \Sigma_{S,n}$. 
The \'etale $\Q_p$-algebra $K \otimes_{\mathbb{Q}}\mathbb{Q}_p$ determines the decomposition of the prime $p$ in $K$. 
The choice $p = \infty$ with $\Q_p = \R$ is permitted, in which case the splitting condition
determines the signature of $E$.

We will need the following hypotheses on the distribution of $S_d$-fields.
\begin{hypothesis}\label{assumption1} Writing
\[
N_{d}(X,\Sigma_{S,d}) := \# \{ F \in \Sigma_{S,d}: F \in \mathcal{F}_{S_d}, \ |\Disc(F)| < X \},
\]
assume that the following facts hold.
\end{hypothesis}

\begin{enumerate}[label=\textbf{\ref{assumption1}.\Alph*}]

\item\label{A2} (Asymptotic density with local conditions) For each 
finite set of primes $S$ and set of splitting conditions $\Sigma_{S, d}$,
we have
\begin{equation}\label{eq:asym}
N_{d}(X,\Sigma_{S,d}) \sim
C_{\Sigma_{S,d}}X
\end{equation}
for some constant $C_{\Sigma_{S,d}} > 0$. (The rate of convergence is not assumed to be uniform in $\Sigma_{S, d}$.)

\item \label{A3} (Averaged uniformity estimate) 
We assume a uniform upper bound for $N_{d}(X,\Sigma_{S,d})$, on average as 
$\Sigma$ ranges over tamely ramified ``splitting types'' in certain intervals. 
%

Write $g_{1}, \dots, g_{m}$ 
for representatives of the nontrivial
conjugacy classes in $S_d$.
For each $m$-tuple of coprime positive, squarefree integers $(q_1, q_2, \cdots, q_m)$, all coprime to $d!$, we define 
a {\itshape splitting type} $\Sigma_{(q_1, \cdots, q_m)}$ consisting,
for each prime $p$ dividing one of the $q_k$, of all choices of \'etale $\Q_p$-algebras of dimension $d$ such that the associated inertia group 
is conjugate to $\langle g_{k} \rangle$.
(See Section  \ref{subsec:splitting} for more concerning this notation.)

Then, for each $m$-tuple of positive integers $(Q_1, Q_2, \cdots, Q_m)$, we assume that the estimate
\begin{equation}\label{eq:A2}
\sum_{\Sigma} 
N_{d}(X,\Sigma) \ll X \prod_{k = 1}^m Q_k^{r_{g_k}}
\end{equation}
holds, where the sum is over all $\Sigma = \Sigma_{(q_1, \cdots, q_m)}$ with $q_k \in [Q_k, 2Q_k)$ for each $k$, for
constants $r_{g_k}$ which are required to satisfy the inequality 
\begin{equation}\label{eq:index_bound}
r_{g_k} + \mathrm{ind}(g_{k})-\frac{\mathrm{ind}(g_{k}, h)}{|A|}  < 0
\end{equation}
simultaneously for all nontrivial $h \in A$. The implied constant in \eqref{eq:A2} may depend at most on $d$ and $A$.

\end{enumerate}
We refer ahead to  Section \ref{sec:ant}
for more on the relevance of the inertia groups and the index functions described above.

Although Hypothesis \ref{A3} logically depends on both $d$ and $A$, implicitly we regard it as a hypothesis on $d$ alone, because
the shape of \eqref{eq:index_bound}
lends itself to a proof for all abelian groups $A$ simultaneously. This is what we will do for $d = 3, 4, 5$ in Section \ref{sec:A3}.

Our general main theorem is the following.
\begin{theorem}\label{thm:main_general}
If Hypothesis \ref{assumption1} holds for an integer $d \geq 3$ and an abelian group $A$, then the Malle conjecture (\ref{mc1}) is true
for $S_d \times A$.
\end{theorem}

We prove this in Section \ref{maintheoremproof}.

Hypothesis \ref{A2} is known for $d = 3$ by Davenport-Heilbronn \cite{DH} and for $d = 4, 5$ by Bhargava \cite{Bha1, Bha2}, and
conjectured for $d > 5$ by Bhargava. (For $d = 3$, see \cite{BST} or \cite{TT} for an explicit treatment of the local conditions.)

We will prove Hypothesis \ref{A3} for $d = 3, 4, 5$ and all $A$ in Section \ref{sec:A3}. Although the required bound is a bit unwieldy, we will see that it is quantitatively weak. 
For example, as can be seen from the proof, one sufficient condition would be a bound of the form
\begin{equation}
\sum_{q_1 \in [Q_1, 2Q_1)} 
\sum_{q_2 \in [Q_2, 2Q_2)} 
\# \{ F \in \mathcal{F}_{S_d} : \ |\Disc(F)| < X, \ q_1 q_2 ^2 \mid \Disc(F) \} \ll_\epsilon X Q_1^{\epsilon} Q_2^{-\eta},
\end{equation}
assumed for every $\epsilon > 0$ and for any $\eta > 0$. If Hypothesis \ref{A2} is ever proved for any $d > 5$, it seems plausible that
the same techniques might also prove Hypothesis \ref{A3}.

\medskip
Our proof also yields an asymptotic formula (of the same order of magnitude) 
if a finite set of splitting conditions are imposed upon the $S_d \times A$-fields being counted.
(See Remark \ref{rk:splitting}.)
In particular, such conditions are needed in 
our application \eqref{eq:claim_colmez}
to the Colmez conjecture, where $F$ is assumed to be totally real and $K$ is imaginary quadratic. 

Finally, although we only state our results for $S_d \times A$, our techniques should yield a proof of Malle's conjecture
for $H \times A$ for any transitive subgroup $H \subseteq S_d$, provided that 
$da_H > |A|a_A$ and that an analogue of Hypothesis \ref{assumption1} holds.

\subsection{The distribution of abelian fields}
Let $A$ be an abelian group, and let $\Sigma_{S,A} = \Sigma_{S, |A|}$ be a finite set of splitting conditions, with 
\[
N_{A}(X,\Sigma_{S,A}) := \# \{ K \in \Sigma_{S,A}: F \in \mathcal{F}_{A}, \ |\Disc(K)| < X \}.
\]

\begin{theorem}\label{thm:abelian}
The following are true:
\begin{enumerate}
\item 
We have
\begin{equation}\label{eq:abelian1}
N_A(X,\Sigma_{S,A}) \sim C_{\Sigma_{S, A}} X^{a_A} (\log X)^{b_A}
\end{equation}
with constants defined as follows: 
\begin{itemize}
\item $a_A := \big(|A| (1 - p^{-1}) \big)^{-1}$, where $p$ is the smallest prime divisor of $|A|$;
\item $C_{\Sigma_{S, A}}$ is a constant depending on $\Sigma_{S, A}$ (and on $A$); 
\item $b_A = b(A) - 1$, with $b(A)$ as in \eqref{mc1}.
\end{itemize}
\item
Suppose, for an arbitrary integer $q$, that $\Sigma_{S, A}$ includes a restriction that
every $K \in \Sigma_{S, A}$ satisfies $q \mid \Disc(K)$. Then we have
\begin{equation}\label{eq:abelian2}
N_A(X,\Sigma_{S,A}) \ll C^{\omega(q)} \left(\frac{X}{q}\right)^{a_A} (\log X)^{b_A}
\end{equation}
for an absolute constant $C$, uniformly in $q$ and $\Sigma_{S, A}$. In particular, for $q$ squarefree, we have
\begin{equation}\label{eq:abelian3}
 C_{\Sigma_{S, A}} \ll C^{\omega(q)}/q.
 \end{equation}
 \end{enumerate}
\end{theorem}
The asymptotic formula \eqref{eq:abelian1} is studied in \cite{wright, Wood10}, and is a direct corollary of Corollary $2.8$ and Lemma $2.10$ in \cite{Wood10};
the uniformity estimate
\eqref{eq:abelian2} is \cite[Theorem 4.13]{W21}; and \eqref{eq:abelian3} follows from the implication
$q \mid \Disc(K) \implies q^{1/a_A} \mid \Disc(K)$.

\subsection{Discussion} The proof closely follows that of the fourth-named author in \cite{W21}.
In \cite{W21}, this author proved Malle's conjecture for $G \times A$ with $G = S_3, S_4, S_5$ and $A$
abelian with order coprime to $2$, $6$, and $30$ respectively. 
There are some technical difficulties 
when the orders of $G$ and $A$ have a common factor. Firstly, given a pair $(F, K)$, it is possible for
$F^c$ and $K$ to intersect nontrivially, and we must bound from above the number of pairs for which this occurs. Secondly, stronger uniformity
estimates for $G$-extensions are now required to show Hypothesis \ref{A3}. We resolve both issues here.

The basic strategy, based on that in \cite{W21}, is as follows. The problem of counting the compositum of number fields by discriminants can be reduced to the problem of
counting pairs $(F, K)$ with $|\Disc(FK)| < X$. We ``almost'' have $|\Disc(FK)| = |\Disc(F)|^{|A|} |\Disc(K)|^d$; in fact
(see Lemma \ref{lem:Delta}) we have
\[
|\Disc(FK)| = |\Disc(F)|^{|A|} |\Disc(K)|^d \Delta(F, K)^{-1}
\]
for an integer $\Delta(F, K)$ divisible only by those primes which ramify in $F$ and $K$, and which we
compute explicitly in Section \ref{verify} (apart from a bounded factor corresponding to any wildly ramified primes).

If we had $|\Disc(FK)| = |\Disc(F)|^{|A|} |\Disc(K)|^d$, then a Tauberian-type result would suffice to deduce asymptotics
for the $FK$ from those for $F$ and $K$. Such a result in the necessary form was proved in \cite[Section 3]{W21},
which we use here.

To incorporate the correction term $\Delta(F, K)$ we break into two cases. When $\Delta(F, K)$ is small,
we incorporate asymptotics (as stipulated in Hypothesis \ref{A2}) 
for the number of pairs $(F, K)$ satisfying relevant splitting conditions. For $\Delta(F, K)$ large, we prove
a ``tail estimate'', bounding the total contribution of such pairs $(F, K)$.

When we consider abelian groups $A$ with small prime divisors, we need to address extra difficulties. The first issue could happen only when the abelian group $A$ has even order. For such an abelian group $A$, an $A$-extension $K$ could have a common subfield with an $S_d$-extension for arbitrary $d$. Therefore we need to get rid of such cases in enumerating all possible pairs. We resolve this issue in Section \ref{sec:zero_density}. 

The second issue is more serious. Roughly speaking, for a fixed $S_d$, in order to prove the theorem where we pair with a particular abelian group $A$, it can be observed in Lemma $5.1$ in \cite{W21} that when the minimal prime divisor of $A$ tends to be smaller, we are required to show a better ``tail estimate". The situation gets even worse when the order of an abelian group $A$ has a non-trivial common factor with $|S_d|$, which is $d!$. We make this relation explicit in Section \ref{sec:A3}. Both aspects point to the difficulty that we need to give an extremely good ``tail estimate" in order to prove the theorem. 

Our innovation here, which we use to streamline and strengthen the proof of this tail estimate, is an {\itshape averaging argument}.
As one example, it would be useful here if we could prove, uniformly for all squarefree $q$, that
\begin{equation}\label{eq:conj_uniform}
\# \{ K \ : [K : \Q] = 3, \ |\Disc(K)| < X, \ q \mid \Disc(K) \} \ll \frac{X}{q^{1 - \epsilon}}.
\end{equation}
Such an estimate is far from known, and in \cite[Section 4]{W21} the fourth author established some {\itshape uniformity estimates}
as a substitute.

In this paper we will modify the method of \cite{W21} to incorporate results like \eqref{eq:conj_uniform} {\itshape on average} 
over dyadic intervals $q \in [Q, 2Q]$, which in general is much easier. For example, the averaged analogue of
\eqref{eq:conj_uniform} is essentially
\begin{equation}
\# \{ K \ : [K : \Q] = 3, \ |\Disc(K)| < X, \ q \mid \Disc(K) \textnormal{ for some } q \in [Q, 2Q]\} \ll X,
\end{equation}
which follows simply from the Davenport-Heilbronn theorem (upon ignoring the divisibility condition).
Overall our required input, given as 
Hypothesis \ref{A3}, is a bit unwieldy to state -- 
but much easier to prove. 

\subsection{Subsequent work.} Some related works were carried out after the initial version of this paper was completed. These include the following:

Perhaps most prominently, Alberts, Lemke Oliver, the fourth author, and Wood \cite{ALOWW} developed a theory of `inductive counting methods',
with a number of related results in a much more general setting. Forthcoming work by the same authors will develop these ideas further. In addition,
Alberts and Bucur \cite{AB} developed a complementary approach to that of \cite{ALOWW} using multiple Dirichlet series, again with a number of new counting results.

Koymans, Lemke Oliver, Sofos, and the second author \cite{KLOST} obtained an asymptotic formula counting $6$-torsion in the class groups of quadratic fields, along with
one counting degree $12$ $D_6$-number fields. Roughly speaking, the proofs amount to counting $S_3 \times C_2$ fields by other invariants, so there are natural similarities
in the methodology. 

Finally, Dixit and Pasupulati \cite{DP} obtained results on the genus statistics of $S_3 \times C_q$ fields, where $q \neq 3$ is a prime, and in the case $q = 2$ they additionally
obtained a power savings in the error term on Malle's conjecture.
    
\section{Background from Algebraic Number Theory}\label{sec:ant}
Let $K$ be a number field of degree $n$ over $\Q$, not necessarily Galois, with $n$ different embeddings $\lambda_i: K\to \bar{\Q}$ for $i = 1, \cdots, n$. The absolute Galois group $G_{\Q}$ acts 
on these $n$ embeddings by composition, namely $\sigma (\lambda_i) =\sigma \circ \lambda_i$. Since this action is a permutation of these embeddings, we get a homomorphism
$$\rho: G_{\Q} \to S_n.$$
The kernel $\Ker(\rho)$ is the subgroup of $G_{\Q}$ that fixes every embedding of $K$, therefore fixes the Galois closure $K^c$ of $K/\Q$. 
The image $\rho(G_{\Q})\subset S_n$ is a transitive subgroup of $S_n$, and as an abstract group is the Galois group $\Gal(K^c/\Q)$.
 
Equivalently, if we number all the embeddings so that $\lambda_1= id$, then $H:=\rho(\Stab(\lambda_1)) \subset \Gal(K^c/\Q)=\rho(G_{\Q})$ 
is the subgroup fixing $K$.
One obtains 
an embedding $\Gal(K^c/\Q) \hookrightarrow S_n$ isomorphic to that induced by $\rho$,
if one starts with the Galois group $\Gal(K^c/\Q)$ and considers its left multiplication action on the left cosets of $H$. 

%
%
%
    
It is very useful to consider $\Gal(K^c/\Q)$ as a permutation group to study the splitting of primes in non-Galois fields. 
The following classical lemma can be found in \cite{Japanese, Neukirch, Melanie}.
    \begin{lemma}[]
    	Given a degree $n$ number field $K$ with $\Gal(K^c/\Q)\subset S_n$ and an arbitrary prime $p$, and let $D$ be the decomposition group and $I$ be the inertia group at a prime $\mathfrak{p}$ above $p$. Suppose that 
    	$$p = \prod_{1\le i\le r} \mathfrak{p}^{e_i}_i, \quad\quad  \Nm(\mathfrak{p}_i) = p^{f_i},$$
    	for $i = 1, 2, \cdots, r$, then $D$ has $r$ orbits each of which is of length $e_if_i$ and decomposes into $f_i$ $I$-orbits of length $e_i$. 
    \end{lemma}
    
   Notice that if $p$ is tamely ramified in $K$, then by \cite[Theorem III.2.6]{Neukirch} we have 
   \[v_p(\Disc(K)) = \sum_i f_i(e_i-1) = \sum_{i} e_i f_i  - \sum_i f_i = n -\sum_i f_i,\] where by the
    above lemma $\sum_i f_i$ is the number of orbits of the inertia group $I_p$ acting on the right cosets.
  And since $I_p =\langle g \rangle$ is generated by a single element when $p$ is tamely ramified, the orbits of the inertia group $I_p$ are exactly the orbits of $g\in \Gal(K^c/\Q) \subset S_n$.     Following \cite{Mall1} we define
    $$\ind(g):=n- |g| , \ g\in S_n,$$
    where $|g|$ is the number of cycles of $g$ as a permutation element in $S_n$. So we have the following well-known lemma, see for example \cite{Mall1, Melanie}.
    
    \begin{lemma}\label{lem:disc-index}
    Given a number field $K$ of degree $n$ that is tamely ramified at $p$, we have
    $$v_p(\Disc(K)) = \ind(g),$$
    where $I_p = \langle g \rangle \subset \Gal(K^c/\Q)\subset S_n$ is the inertia group of $K^c$. 
    \end{lemma}    
    This gives a very convenient combinatorial description to study the discriminant at tamely ramified primes. 
    
    Now given a pair of number fields $(F, K)$ that are linearly disjoint with degrees $m$ and $n$ respectively, so that $[FK:\Q] = [F:\Q][K:\Q] = mn$, then we can study the discriminant of $FK/\Q$ at tamely ramified primes by the above approach. By linear disjointness, the embeddings of $FK\to \bar{\Q}$ are in bijection with pairs of embeddings $(\lambda_i, \delta_j)$ where $\lambda_i: F\to \bar{\Q}$ and $\delta_j: K\to \bar{\Q}$ for $1\le i\le m$ and $1\le j\le n$ are embedding of $F$ and $K$ respectively. So we get a permutation structure of $G_{\Q} \to S_m\times S_n\hookrightarrow S_{mn}$ by its action on $(\lambda_i, \delta_j)$. Therefore we get
    $\Gal(K^cF^c/\Q)\subset \Gal(F^c/\Q)\times \Gal(K^c/\Q)$. 
    
    Then Lemma \ref{lem:disc-index} applies to the compositum $FK$ as well, with
    \[
    v_p(\Disc(FK)) = \ind(g, h),
    \]
    where $g$ and $h$ generate the inertia groups of $F^c$ and $K^c$ respectively, and where we write
    $\ind(g, h)$ for the index of $(g, h)$ with respect to the above permutation representation in $S_{mn}$.
    
    The following lemma gives a general description of the discriminant of $KF$ for an arbitrary pair of linearly disjoint $(K, F)$ based on the permutation structure of $\Gal(K^cF^c/\Q)$. 
    The statement incorporates several lemmas from 
    \cite[Section 2]{W21}, and here we provide a combined statement for the convenience of the reader. 
   
    \begin{lemma}\label{lem:Delta}
    Given  linearly disjoint number fields $K$ and $F$ of degrees $m$ and $n$, denote
    \begin{equation}\label{eq:def_Delta}\Delta_p(K, F) = \frac{\Disc_p(K)^n \Disc_p(F)^m}{\Disc_p(KF)}     \end{equation}
    with $\Disc_p(\cdot) := p^{v_p(\Disc(\cdot))}$.
    Then the following are true:
    \begin{itemize}
        	\item
    	The function $\Delta_p(K, F)$ depends only on the local \'etale algebras $K\otimes_{\Q} \Q_p$ and $F\otimes_{\Q} \Q_p$.
        \item
        The function $\Delta_p(K, F)$ is an integer, with $\Delta_p(K, F) = 1$ if $p$ is unramified in $K$ or $F$.
    	\item 
    	At $p \nmid [(KF)^c:\Q]$, so that $K$ and $F$ are unramified or tamely ramified
	with $I_p(K) = \langle g \rangle \subset S_m$ and $I_p(F) = \langle h \rangle \subset S_n$, define
    	\begin{equation}\label{eq:def_Delta2}
	 \Delta(g, h) := v_p(\Delta_p(K, F)) = n\cdot \ind(g) + m\cdot \ind(h) - \ind(g, h).
	 \end{equation}
	 Then we have
	\begin{align*}
	\Delta(g, h) = & \ n \sum_i \left( |c_i| - 1\right) + m \sum_j \left( |d_j| - 1\right) -\big( mn -\sum_{i,j} \gcd(|c_i|, |d_j|)\big),
	\end{align*}
    	where $g = \prod_{i} c_i$ and $h = \prod_j d_j$ are the cycle structures. 
	\item
	We have $\prod_{p \mid [(KF)^c:\Q]} \Delta_p(K, F) = O_{n, m}(1)$.
    \end{itemize}    
    \end{lemma}
    
    \begin{remark}
        As another (equivalent) way to compute $\Delta_p(K, F)$, note that if
    	\begin{align*}
v_p(\Disc(F)) = \sum_{i = 1}^{g_F} f_i (e_i - 1), \ \ \ v_p(\Disc(K)) = \sum_{j = 1}^{g_k} f_j (e_j - 1),
\end{align*}
where the $(e_i, f_i)$ and $(e_j, f_j)$ are the ramification and inertial degrees of the prime ideals of $K$ and $F$ above $p$ respectively, 
then for $p \nmid mn$ we have
\begin{align*}
v_p(\Disc(FK)) = \sum_{i = 1}^{g_F} \sum_{j = 1}^{g_k} f_i f_j \gcd(e_i, e_j) (\textnormal{lcm}(e_i, e_j) - 1).
\end{align*}

    \end{remark}
    
    Finally, we prove that pairs $(F, K)$ are in bijection with their composita:
    
    \begin{lemma}\label{lem:composita}
    	Given arbitrary $d>2$ and any abelian group $A$, the map $\phi : (F, K) \mapsto FK$ induces a bijection between
    	$$\{ (F, K) \mid \Gal(F^c/\Q)= S_d,[F:\Q]=d, \Gal(K/\Q)= A, \Gal(F^cK/\Q)=S_d\times A\subset S_{d|A|} \}$$
    	and
    	$$\{  L \mid [L:\Q]=d|A|, \Gal(L^c/\Q)=S_d\times A\subset S_{d|A|} \}.$$
    \end{lemma}
    \begin{proof}
    Firstly, we show $\phi$ is surjective. For each $L$ with $\Gal(L^c/\mathbb{Q}) = S_d\times A\subset S_{d|A|}$, we can write $L = (L^c)^H$ where $H = S_{d-1}\times e$. We define the inverse function $\psi (L) = (F, K)=( (L^c)^{S_{d-1}\times A}, (L^c)^{S_d\times e})$.
    
    Next, we show $\phi$ is injective. It suffices to show that there is a unique subfield $F$ of $L$ with $\Gal(F^c/\mathbb{Q})= S_d$ and a unique subfield $K$ with $\Gal(K/\Q) = A$. Given $H = S_{d-1}\times e$, we claim that $H_2 = S_d\times e$ is the only subgroup containing $H$ with $G/H_2 \simeq A$.
    Indeed, the kernel of any surjection $S_d \times A \to A$ must contain $A_d \times e$. If this kernel
    also contains $H$, then it must contain $H_2$ and hence equal $H_2$ as claimed. 
         This gives the uniqueness of $K$. 
    
    On the other hand, we claim that the only surjection $f: S_d\times A\to S_d$ is $\text{id}\times e$. 
   This shows that only one $S_d$ Galois extension exists in $L^c$, thus that only one isomorphism class of $S_d$ degree $d$ subextensions exists in $L^c$. 
    Denote $M_1:= f(S_d\times e)$ and $M_2:=f(e\times A)$ to be subgroups of $S_d$; then $M_1\cdot M_2 = S_d$ since $f$ is surjective, and $M_1$ commutes 
    with $M_2$ since $S_d\times e$ commutes with $e\times A$ in $S_d\times A$. Therefore $M_2$ is contained in the center of $S_d$, which is trivial, so $M_1 = S_d$.   	

    \end{proof}

\section{Index computations and Hypothesis \ref{A3}}\label{sec:A3}
In this section we study the index function on $S_d \times A$, and establish that Hypothesis
\ref{A3} holds for $d = 3, 4, 5$ and all $A$.

\begin{lemma}\label{lem:index-compare}
	For any $(g, h) \in G\times H$, we have
	$$\ind(g, e) = \ind(g)|H| \le \ind(g, h).$$
	Equality holds if and only if
	$$ \ord(h) \mid \gcd_i(|c_i|),$$
	where $g = \prod_{i} c_i$ is the disjoint union of cycles $c_i$.
\end{lemma}
\begin{proof}
Write $g = \prod_i c_i$ and $h = \prod_j d_j$ as disjoint unions of cycles. Then, 
	the number of cycles in $(g, h)$ is
	$$\sum_{i,j } \gcd (|c_i|, |d_j|),$$
	and the number of cycles in $(g, e)$ is
	$$\sum_{i,j } |d_j|.$$
	The first inequality follows immediately. We obtain equality if and only if $|d_j| \mid |c_i|$ for every $i$ and
	$j$, or equivalently if $\ord(h) \mid |c_i|$ for every $i$, yielding the second statement.
	
\end{proof}
We now turn our attention to \eqref{eq:index_bound}.
Given the definition 
$\Delta(g, h) = \ind(g)\cdot |A| + d\cdot \ind(h) -\ind(g, h)$, the following lemma is an immediate consequence of Lemma \ref{lem:index-compare}.

\begin{lemma}\label{lem:inequ}
Let $(g, h)$ be any nonidentity element of $S_d\times A$. Then, we have
\begin{equation}\label{eq:inequ}
\ind(g) - \frac{\ind(g, h)}{|A|}
= 
\frac{\Delta(g, h)}{|A|}  - \frac{\ind(h)\cdot d}{|A|} \leq 0.
\end{equation}
Moreover, writing  $g= \prod_{i} c_i$, equality holds in \eqref{eq:inequ} if and only if $\ord(h)| \gcd_i( |c_i|)$.

In particular, for $d = 3, 4, 5$, equality can only occur for $h = e$ or in the following cases:
\begin{itemize}
\item
For $d = 3$, $g = (123)\in S_3$ (up to conjugacy) and $3 \mid |A|$;
\item
For $d = 4$, $g = (12)(34), (1234)\in S_4$ (up to conjugacy) and $2 \mid |A|$;
\item
For $d = 5$, $g = (12345) \in S_5$ (up to conjugacy) and $5 \mid |A|$.
\end{itemize}
\end{lemma}

Note that when $\ind(g) - \frac{\ind(g, h)}{|A|}$ is nonzero, it is necessarily bounded above by $- \frac{1}{|A|}$.

\subsection{Examples}
To help the reader visualize the results, we present an explicit computation of $\Delta(g, h)$ in the following five cases:
$G \times A = S_3 \times C_2$; $S_3 \times C_3$; $S_4 \times C_2$; $S_5 \times C_2$; $S_5 \times C_5$. 

For each prime $p$ tamely ramified in $FK$ (so, in particular, each $p > 5$), the following tables present the following data:
\begin{itemize}

\item The possible ramified  ``splitting types'' of $p$ in $\mathcal{O}_F$, listing the inertial degree and ramification index
of each of the primes $\frak{p}$ of $\mathcal{O}_F$ over $p$. 

The formatting and list is as in \cite[Lemma 20]{HCL4}; for example, writing that $p$ has splitting type $(1^2 12)$ means that
$p \mathcal{O}_F = \frak{p}_1^2 \frak{p}_2 \frak{p}_3$
for distinct primes $\frak{p}_1$, $\frak{p}_2$, and $\frak{p}_3$ of $F$ with inertial degrees $1$, $1$, and $2$ respectively.

\item The splitting type of $p$ in $FK$, assuming that $p$ has the designated splitting type in $F$ and ramifies in $K$.

\item A generator $g$ of the inertia group $I_{F, p} \subseteq S_d$ at $p$, well defined up to conjugacy.
(In the examples here $A$ is cyclic of prime order; since we assume that $p$ ramifies in $K$, the corresponding inertia group will be all of $A$.)

\item The $p$-adic valuation of $\Disc(F)$, which can be computed from Lemma \ref{lem:disc-index} or
from the splitting type. 

\item The $p$-adic valuation of $\Disc(FK)$, computed in the same way.

\item The ``discrepancy'' $\Delta(g, h)$, defined in general in Lemma \ref{lem:Delta}, and here equal to
\begin{align}\label{discrepancy}
\Delta(g, h) = |A| \cdot v_p(\Disc(F)) + d \cdot v_p(\Disc(K)) - v_p(\Disc(FK)).
\end{align}
Recall that in Lemma \ref{lem:inequ} we proved that $\Delta(g, h) \leq \ind(h) \cdot d$.
\end{itemize}

\begin{table}[h!]\caption{Discriminant valuations for $S_3 \times C_2$ ($\Delta(g, h) \leq 3$)}
\begin{center}
\begin{tabular}{c|c|c|c|c|c}
Splitting type &  Splitting in $FK$ & Generator $g$ & $v_p(\Disc(F))$ & $v_p(\Disc(FK))$ & $\Delta(g, h)$  \\ \hline
$(1^2 1)$ & $(1^2 1^2 1^2)$, $(1^2 1^2 1^2)$ & $(12)$ & $1$ & $3$ & $2$  \\ 
$(1^3)$ & $(1^6)$ & $(123)$ & $2$ & $5$ & $2$ \\ 
\end{tabular}
\end{center}
\end{table}

\begin{table}[h!]\caption{Discriminant valuations for $S_3 \times C_3$ ($\Delta(g, h) \leq 6$)}
\begin{center}
\begin{tabular}{c|c|c|c|c|c}
Splitting type &  Splitting in $FK$ & Generator $g$ & $v_p(\Disc(F))$ & $v_p(\Disc(FK))$ & $\Delta(g, h)$  \\ \hline
$(1^2 1)$ & $(1^6 1^3)$ & $(12)$ & $1$ & $7$ & $2$  \\ 
$(1^3)$ & $(1^3 1^3 1^3)$ & $(123)$ & $2$ & $6$ & $6$ \\ 
\end{tabular}
\end{center}
\end{table}

\begin{table}[h!]\caption{Discriminant valuations for $S_4 \times C_2$ ($\Delta(g, h) \leq 4$)}
\begin{center}
\begin{tabular}{c|c|c|c|c|c}
Splitting type &  Splitting in $FK$ & Generator $g$ & $v_p(\Disc(F))$ & $v_p(\Disc(FK))$ & $\Delta(g, h)$   \\ \hline
$(1^2 11)$, $(1^2 2)$ & $(1^2 1^2 1^2 1^2)$, $(1^2 1^2 2^2)$ & $(12)$ & $1$ & $4$ & $2$  \\ 
$(1^2 1^2)$, $(2^2)$ & $(1^2 1^2 1^2 1^2)$, $(2^2 2^2)$ & $(12)(34)$ & $2$ & $4$ & $4$ \\ 
$(1^3 1)$ & $(1^6 1^2)$ & $(123)$ & $2$ & $6$ & $2$ \\ 
$(1^4)$ & $(1^4 1^4)$ & $(1234)$ & $3$ & $6$ & $4$ \\ 
\end{tabular}
\end{center}
\end{table}
\begin{table}[h!]\caption{Discriminant valuations for $S_5 \times C_2$ ($\Delta(g, h) \leq 5$)}
\begin{center}
\begin{tabular}{c|c|c|c|c|c}
Splitting type &  Splitting in $FK$ & Generator $g$ & $\small{v_p(\Disc(F))}$ & $\small{v_p(\Disc(FK))}$ & $\Delta(g, h)$   \\ \hline
\scriptsize{$(1^2 111)$, $(1^2 12)$, $(1^2 3)$} & \scriptsize{$(1^2 1^2 1^2 1^2 1^2)$, $(1^2 1^2 1^2 2^2)$, $(1^2 1^2 3^2)$} & $(12)$ & $1$ & $5$ & $2$  \\ 
$(1^2 1^2 1)$, $(2^2 1)$ & $(1^2 1^2 1^2 1^2 1^2)$, $(2^2 2^2 1^2)$ & $(12)(34)$ & $2$ & $5$ & $4$ \\ 
$(1^3 11)$, $(1^3 2)$ & $(1^6 1^2 1^2)$, $(1^6 2^2)$ & $(123)$ & $2$ & $7$ & $2$ \\ 
$(1^3 1^2)$ & $(1^6 1^2 1^2)$ & $(123)(45)$ & $3$ & $7$ & $4$ \\ 
$(1^41)$ & $(1^4 1^4 1^2)$ & $(1234)$ & $3$ & $7$ & $4$ \\ 
$(1^5)$ & $(1^{10})$ & $(12345)$ & $4$ & $9$ & $4$ \\ 
\end{tabular}
\end{center}
\end{table}

\begin{table}[h!]\caption{Discriminant valuations for $S_5 \times C_5$ ($\Delta(g, h) \leq 20$)}
\begin{center}
\begin{tabular}{c|c|c|c|c|c}
Splitting type &  Splitting in $FK$ & Generator $g$ &  $\small{v_p(\Disc(F))}$ & $\small{v_p(\Disc(FK))}$  & $\Delta(g, h)$  \\ \hline
\scriptsize{$(1^2 111)$, $(1^2 12)$, $(1^2 3)$} & \scriptsize{$(1^{10} 1^5 1^5 1^5)$, $(1^{10} 1^5 1^5 2^5)$, $(1^5 1^5 3^5)$} & $(12)$ & $1$ & $21$ & $4$  \\ 
$(1^2 1^2 1)$, $(2^2 1)$ & $(1^{10} 1^{10} 1^5), (2^{10} 1^5)$ & $(12)(34)$ & $2$ & $22$ & $8$ \\ 
$(1^3 11)$, $(1^3 2)$ & $(1^{15} 1^5 1^5)$, $(1^{15} 2^5)$ & $(123)$ & $2$ & $22$ & $8$ \\ 
$(1^3 1^2)$ & $(1^{15} 1^{10})$ & $(123)(45)$ & $3$ & $23$ & $12$ \\ 
$(1^41)$ & $(1^{20} 1^5)$ & $(1234)$ & $3$ & $23$ & $12$ \\ 
$(1^5)$ & $(1^5 1^5 1^5 1^5 1^5)$ & $(12345)$ & $4$ & $20$ & $20$ \\ 
\end{tabular}
\end{center}
\end{table}

\subsection{Verification of Hypothesis \ref{A3} for $S_d$ with $d=3,4,5$}\label{verify}

We proceed now to the verification of Hypothesis \ref{A3}. 
We begin with two field counting lemmas. The first is a combination of
\cite[Theorem 1.3]{TT} and \cite[Theorems 4.1 and 5.1]{EPW}.
\begin{lemma}\label{lem:power_saving}
Let $d \in \{ 3, 4, 5 \} $. For squarefree $q$ we have

\begin{equation}\label{eq:quant}
 \# \{ F \in \mathcal{F}_{S_d} : \ |\Disc(F)| < X, \ q \mid \Disc(F) \} =
 C X \prod_{p \mid q} \left( p^{-1} + O(p^{-2}) \right) + O\left(
   X^{1 - \alpha} q^{\beta} \right),
 \end{equation}
 for constants $C > 0$, $\alpha > 0$ and $\beta > -1$ depending on $d$ (as does the implied constant).
\end{lemma}

\begin{lemma}\label{lem:bound_smallQ}
When $Q_1 Q_2 \ll X^{1/40}$, we have
that
\[
\sum_{q_1 \sim Q_1}
\sum_{q_2 \sim Q_2} 
\# \{ F \in \mathcal{F}_{S_5}, \ |\Disc(F)| < X, \ q_1 q_2^2 \mid \Disc(F) \} \ll XQ_1^{\epsilon} Q_2^{-1 + \epsilon}.
\]
\end{lemma}
Here, and throughout this section, we write $q \sim Q$ as a shorthand for $q \in [Q, 2Q)$. Throughout all $q_i$ are assumed squarefree
and coprime to each other.
\begin{proof}
This follows from \cite[(27)]{Bha2}, as we now briefly recall.

In \cite{Bha2}, fields $F \in \mathcal{F}_{S_5}$ are parametrized by points in a certain $40$-dimensional lattice
$V_\Z$, up to the action of a group $G_\Z$, and which satisfy certain conditions, including a `maximality' condition $\pmod{p^2}$ for every prime $p$.
The condition that $q_1 q_2^2 \mid \Disc(F)$
corresponds to a set of congruence conditions modulo $q_1 q_2$ in $V_\Z$, and the density $\delta(q_1, q_2)$ of these conditions
is seen to be $\ll q_1^{-1 + \epsilon} q_2^{-2 + \epsilon}$ by the tables in \cite[Section 12]{HCL4}. 

We keep also the condition that $v \in V_\Z$ be `irreducible' (satisfied for any quintic field), but drop the remaining
conditions, including maximality. Then, \cite[(27)]{Bha2} provides an upper bound for the number of such lattice points, subject to the condition
that $q_1 q_2 = O(X^{1/40})$, yielding 
\[
\# \{ F \in \mathcal{F}_{S_5}, \ |\Disc(F)| < X, \ q_1 q_2^2 \mid \Disc(F) \} \ll
\delta(q_1, q_2) X + O\big(\delta(q_1, q_2) q_1 q_2 X^{39/40} \big).
\]
The error term is $O(q_1^\epsilon q_2^{-1 + \epsilon} X^{39/40})$, bounded by the 
main term when $q_1 q_2 \ll X^{1/40}$.
Summing over $q_1 \sim Q_1$ and $q_2 \sim Q_2$ yields the stated bound.
\end{proof}

We also note that in proving Hypothesis \ref{A3} it is possible to combine splitting types. 
Suppose that $\gamma_{1}, \cdots, \gamma_{m}$ are disjoint sets whose union is
a set of the representatives of the nontrivial conjugacy classes in $S_d$. We then
define our splitting types $\Sigma_{(q_1, \cdots, q_m)}$, such that for each $p \mid q_k$
we insist that the associated inertia group have generator conjugate to any element of $\gamma_{k}$.
If we define $N_d(X, \Sigma)$ analogously, we ask that
\begin{equation}\label{eq:A2_combined}
\sum_{\Sigma} 
N_{d}(X,\Sigma) \ll X \prod_{k = 1}^m Q_k^{r_{\gamma_k}},
\end{equation}
as the direct analogue of \eqref{eq:A2}, where $r_{\gamma_k}$ is required to satisfy 
\eqref{eq:index_bound} for each $g \in \gamma_{k}$. It is essentially immediate to check 
that \eqref{eq:A2_combined} implies \eqref{eq:A2}.

\medskip

For $d = 3, 4$, we achieve an essentially sharp quantitative bound.
For an $S_d$-field $F$ ($d = 3, 4$), the Galois closure $F^c$ contains a 
subfield $E$ of degree $d - 1$ inside $F^c$, unique up to isomorphism, called 
the (quadratic or cubic) {\itshape resolvent field} of $F$.
We have
\begin{equation}\label{eq:resolvent_disc}
\Disc(F) = \Disc(E) f^2
\end{equation}
for a positive integer $f$, squarefree apart from a possible factor of $2$ when $d = 4$ or $3$ when $d = 3$, and divisible
precisely by those primes which have the following splitting types in $F$: for $d = 3$, $(1^3)$; for
$d = 4$,  $(1^2 1^2)$, $(2^2)$, or $(1^4)$.

In this situation,  the number of $F$ with resolvent $E$ and which satisfy
\eqref{eq:resolvent_disc} for a given $f$ is $O(\# \Cl(E)[3] \cdot 2^{\omega(f)})$ in the cubic case, 
and $O(\# \Cl(E)[2] \cdot 3^{\omega(f)})$ in the quartic case. Moreover, in the case $f = 1$, this
number equals $\frac{ \# \Cl(E)[3] - 1}{2}$ or $\# \Cl(E)[2] - 1$ respectively, so that
Lemma \ref{lem:power_saving} also implies the estimate
\begin{equation}\label{eq:ps2}
\sum_{\substack{ |\Disc(E)| < X \\ q \mid \Disc(E)}} \# \Cl(E)[\ell] \ll
X \prod_{p \mid q} \left( p^{-1} + O(p^{-2}) \right) +
   X^{1 - \alpha} Q^{\beta} 
 \end{equation}
 in both cases, with $\ell = 3$ or $2$ respectively.
This machinery also leads to {\itshape local uniformity
estimates}, to be cited in the course of the proof.

This machinery is due largely to Hasse and Heilbronn; for references, see \cite[Lemmas 3.1 and 5.1]{BBP}, \cite{baily}, 
\cite[Section 6]{DW3}, and \cite[Lemma 26]{Bha1}.

For $d = 3$, then, denote
\begin{equation}\label{eq:n3_def}
N_3(X; q_1, q_2) := 
\# \Big\{ F \in \mathcal{F}_{S_3}, \ |\Disc(F)| < X, \ \substack{ p \mid q_1 \implies p \textnormal{ has splitting type } (1^2 1) \\ 
 p \mid q_2 \implies p \textnormal{ has splitting type } (1^3) 
} \Big\}.
\end{equation}
We seek a bound on $\sum_{q_1 \sim Q_1} \sum_{q_2 \sim Q_2} N_3(X; q_1, q_2)$.
     For $d = 3$, we will prove that
     \begin{equation}\label{eq:d3_toprove}
     \sum_{q_1\sim Q_1} \sum_{q_2\sim Q_2}  N_3(X; q_1, q_2) \ll X Q_1^{\epsilon} Q_2^{-1+\epsilon}.
     \end{equation}
     
     Suppose first that $Q_1 Q_2 > X^{\eta}$, for any fixed $\eta > 0$. Then, since
    each $F$ in the sum has at most $O(X^{\epsilon'})$ ramified primes,  
     we have (for any $\epsilon' > 0$) that
     \begin{equation}\label{eq:d3_1}
\small{     \sum_{q_1\sim Q_1} \sum_{q_2\sim Q_2} N_3(X; q_1, q_2) \ll_{\epsilon'} X^{\epsilon'} 
     \sum_{q_2\sim Q_2} 
     \# \Big\{ F \in \mathcal{F}_{S_3}, \ |\Disc(F)| < X, \ 
 p \mid q_2 \implies p \textnormal{ has splitting type } (1^3) 
 \Big\} }.
\end{equation}
A local uniformity estimate, due essentially to Davenport and Heilbronn and appearing more explicitly as \cite[Proposition 6.2]{DW3}
or \cite[Lemma 3.3]{BBP}, states that each summand on the right of \eqref{eq:d3_1} is $\ll X/q_2^{2 - \epsilon'}$.
Therefore, the sum in \eqref{eq:d3_1} is $O(X^{1 + \epsilon'} / Q_2^{1 - \epsilon'}$). The result follows upon choosing $\epsilon' = \frac{\eta \epsilon}{\eta + 1}$.
     
When $Q_1 Q_2 \leq X^{\eta}$, we have
     \begin{equation}\label{eq:d3_bound}
     \begin{aligned}
     \sum_{q_1 \sim Q_1} \sum_{q_2 \sim Q_2} N_3(X; q_1, q_2)
     & \ll \sum_{q_1 \sim Q_1} \sum_{q_2 \sim Q_2} \sum_{\substack{E \\ q_1|\Disc(E)}} \# \Cl_E[3]
     \sum_{\substack{f\le \left(\frac{X}{\Disc(E)}\right)^{1/2} \\ q_2 \mid f}} 2^{\omega(f)}\\
     & <
     \sum_{q_2 \sim Q_2} 
     \sum_{\substack{f \\ q_2 \mid f}} 2^{\omega(f)} 
     \sum_{q_1 \sim Q_1}
     \sum_{\substack{\Disc(E) \leq X/f^2 \\ q_1 \mid \Disc(E)}} \# \Cl_E[3]
     \\ &
     \ll 
     \sum_{q_2 \sim Q_2} 
     \sum_{\substack{f \\ q_2 \mid f}} 2^{\omega(f)} 
     \sum_{q_1 \sim Q_1}
     \bigg( \frac{X}{q_1^{1 - \epsilon} f^{2}} + \frac{X^{1 - \alpha} q_1^\beta}{f^{2 - 2\alpha}} \bigg).
     \end{aligned} 
     \end{equation}
     The sums over $E$ are over quadratic fields, and the last inequality follows by Lemma \ref{lem:power_saving}
     and \eqref{eq:ps2}.
We have that
\[
\sum_{\substack{f \\ q_2 \mid f}} \frac{2^{\omega(f)}}{f^{2 - 2\alpha}} \ll \frac{1}{q_2^{2 - 2\alpha - \epsilon}},
\]
and similarly with the $2 \alpha$ removed. We thus have
\[
     \sum_{q_1 \sim Q_1} \sum_{q_2 \sim Q_2} N_3(X; q_1, q_2)
\ll X Q_1^{\epsilon} Q_2^{-1 + \epsilon} \left( 
1 + X^{-\alpha} Q_1^{\beta + 1} Q_2^{2\alpha}
\right).
\]
Assuming that $\beta > 0$ and $\alpha < \frac{1}{2}$, as we may, the claimed result follows upon choosing $\eta = \frac{\alpha}{\beta + 1}$.
  
  \medskip
  For $d = 4$, we will combine splitting types as outlined above, and write
  \begin{equation}\label{eq:n4_def}
  N_4(X; q_1, q_2) := 
\# \Big\{ F \in \mathcal{F}_{S_4}, \ |\Disc(F)| < X, \ \substack{ p \mid q_1 \implies p \textnormal{ has splitting type } (1^2 11), \ (1^2 2), \ \textnormal{or } (1^3 1) \\ 
 p \mid q_2 \implies p \textnormal{ has splitting type } (1^2 1^2), \ (2^2), \ \textnormal{or } (1^4) \\
} \Big\}.
\end{equation}
The same argument as above now yields the bound
         \begin{equation}\label{eq:d4_toprove}
         \sum_{q_1\sim Q_1} \sum_{q_2\sim Q_2}  N_4(X; q_1, q_2) \ll X Q_1^{\epsilon} Q_2^{-1+\epsilon}.
         \end{equation}
In \eqref{eq:d3_bound} the sums over $E$ are now over
cubic fields instead of quadratic; $\Cl_E[3]$ is replaced with $\Cl_E[2]$ and $2^{\omega(f)}$ with $3^{\omega(f)}$;
the required local uniformity estimate is \cite[Lemma 5.2]{BBP}.
In all other respects the proof is identical. 

\medskip
For $d = 5$, uniformity estimates of comparable strength are not known; however, it was proved in \cite{BCT} that 
\begin{equation}\label{eq:bound_bigQ}
\sum_{q_2 \sim Q_2} 
\# \{ F \in \mathcal{F}_{S_5}, \ |d_F| < X, \ q_2^2 \mid d_F \} \ll X^{39/40 + \epsilon'} + X^{1 + \epsilon'}/Q_2,
\end{equation}
 by means of the Bhargava-Ekedahl geometric sieve \cite{geosieve}. Defining (for squarefree $q_1, q_2$) 
 \begin{equation}\label{eq:n5_def}
N_5(X; q_1, q_2) := 
\# \{ F \in \mathcal{F}_{S_5}, \ |d_F| < X, \ q_1 q_2^2 \mid d_F \},
\end{equation}
we thus claim that
\begin{equation}\label{eq:geosieve_app}
\sum_{q_1 \sim Q_1} 
\sum_{q_2 \sim Q_2} 
N_5(X; q_1, q_2) \ll
X Q_1^{\epsilon} Q_2^{- 1/20 + \epsilon},
\end{equation}
uniformly in $Q_2$. For $Q_1 Q_2 \ll X^{1/40}$, this is a weaker form of Lemma \ref{lem:bound_smallQ}. When $Q_1 Q_2 > X^{1/40}$,
\eqref{eq:bound_bigQ} implies that
\begin{align*}
\sum_{q_1 \sim Q_1} 
\sum_{q_2 \sim Q_2} 
N_5(X; q_1, q_2) & \ll X^{39/40 + 2\epsilon'} + X^{1 + 2\epsilon'}/Q_2 \\
&  \ll (Q_1 Q_2)^{\epsilon} \cdot \left(X^{39/40} + X/Q_2\right) \\
&  \ll (Q_1 Q_2)^{\epsilon} \cdot XQ_2^{-1/20}
\end{align*}
as claimed, as $q_2 < X^{1/2}$ if $N_5(X; q_1, q_2) \neq 0$.

\medskip
{\itshape Comparison to needed results.} The final step is to show that
\eqref{eq:d3_toprove}, \eqref{eq:d4_toprove}, 
\eqref{eq:geosieve_app} are at least as strong as required in Hypothesis \ref{A3}.

By Lemma \ref{lem:inequ}, we require that $r_{g} < 0$ for each $g$ explicitly enumerated there, and that
$r_g <  \frac{1}{|A|}$ for the remaining $g$. 
Combining splitting types as in \eqref{eq:A2_combined}, 
we therefore require in \eqref{eq:d3_toprove}, \eqref{eq:d4_toprove}, 
\eqref{eq:geosieve_app} a bound of $X Q_1^{\alpha_1} Q_2^{\alpha_2}$,
with $\alpha_i < 0$ when the primes dividing $q_i$ may have the splitting types enumerated in Lemma \ref{lem:inequ},
and $\alpha_i < \frac{1}{|A|}$ otherwise. In all three cases we therefore see that a bound of
$X Q_1^{1/|A| - \delta} Q_2^{-\delta}$ suffices for any fixed $\delta$, and in particular that the bounds proved above suffice.

\section{Proof of Theorem \ref{thm:main_general}}\label{maintheoremproof}

For $d = 3, 4, 5$ and an abelian group $A$, we write
	\begin{equation}\label{eq:def2db}
	\Gamma = \Gamma_{A, d} := \{(F,K):~ [F : \Q] = d, \ \Gal(F^c/\Q) = S_d, \ \Gal(K/\Q) = A\}.
	\end{equation}

Then, by Lemma \ref{lem:composita}, we have 
	\begin{equation}\label{eq:def2da}
	N(X, S_d \times A) = \#\{(F,K) \in \Gamma, \Gal(F^cK/\Q)= S_d\times A,  ~ | \Disc(FK) | <X \}.
	\end{equation}
	
	Let $Y > |A|d!$ be a parameter, and let $S_Y$ be the set of all rational primes $p \leq Y$. For any pair $(F,K) \in \Gamma$, define 
	\begin{align*}
	\Disc_Y(FK) := \prod_{p \in S_Y} \Disc_p(FK) \prod_{p \notin S_Y} \Disc_p(F)^{|A|} \Disc_p(K)^{d}
	\end{align*}
	where $\Disc_p(L)$ denotes the $p$-part of the discriminant of a number field $L$. Let 
	\begin{equation}\label{eq:def2da_2}	
	N_Y(X, S_d \times A) := \# \{(F,K) \in \Gamma:~ \Gal(F^cK/\Q) = S_d\times A,  \Disc_Y(FK) < X\}.
	\end{equation}
	
	Then it is immediate from Lemma \ref{lem:Delta} that for 
	$ Y > 0$ we have
	\begin{equation}\label{eq:NY_ineq}
	N_Y(X, S_d \times A) \leq N(X, S_d \times A), 
	\end{equation}
	and that for $0 < Y_1 < Y_2$ we have
	\begin{equation}\label{eq:NYY_ineq}
	N_{Y_1}(X, S_d \times A) \leq  N_{Y_2}(X, S_d \times A).
	\end{equation}

\subsection{Splitting types}\label{subsec:splitting}
	Let $S$ be a finite set of rational primes.
	Let $\Sigma_{S,n}$ consist of a choice of \'etale $\Q_p$-algebra of dimension $n$ for each prime $p \in S$. 
	An \'etale $\Q_p$-algebra $\mathcal{A}$ of dimension $n$ can be expressed uniquely (up to isomorphism) as 
	a direct product of finite extensions $K_{\frak{P}_i}$ of $\mathbb{Q}_{p}$ for $i=1, \ldots, \ell(\mathcal{A})$; that is,
	\begin{align*}
	\mathcal{A} \cong \prod_{i=1}^{\ell(\mathcal{A})} K_{\mathfrak{P}_{i}}, \quad \mathfrak{P}_{i}\cap\mathbb{Q} = (p), \quad
	\sum_{i=1}^{\ell(\mathcal{A})} \left[K_{\mathfrak{P}_{i}}:\mathbb{Q}_{p}\right] = n.
	\end{align*}
	It follows that given $p \in S$, there are only finitely many \'etale $\Q_p$-algebras of dimension $n$, and hence  
	only finitely many choices of sets $\Sigma_{S,n}$. 
	For each number field $L$ of degree $n$, we write $L \in \Sigma_{S,n}$ if 
	$\{L \otimes_{\Q} \Q_p\}_{p \in S}=\Sigma_{S,n}$. 
	
	Let $E_S$ denote the set of all possible pairs $(\Sigma_{S, d}, \Sigma_{S,A})$ as described above.
	For each 
	$\Sigma_{S} \in E_S$, we say that $(F,K) \in \Sigma_{S}$ 
	(or that $(F,K) \in \Sigma_{S} \cap \Gamma$, when $(F, K) \in \Gamma$)
	if $F \in \Sigma_{S, d}$ and $K \in \Sigma_{S,A}$.
	\subsection{Setup: a disjoint union.}\label{sec:disjoint_union}
	We introduce the approximation 
	\begin{equation}\label{def:tildeNY}
	N_Y'(X, S_d\times A):=\#\{  (F, K) \in \Gamma: \Disc_Y(FK)<X \}
	\end{equation}
	to $N_Y(X, S_d\times A)$, 
	which omits the condition that $\Gal(F^cK) = S_d\times A$. We will prove in
	Section \ref{sec:zero_density} that these two functions have the same asymptotic behavior as $X \rightarrow \infty$.

	We have a disjoint union
	\begin{align*}
	&\{(F,K) \in \Gamma:~ \Disc_Y(FK) < X\} =
	\bigsqcup_{\Sigma_{S_Y} \in E_{S_Y}}
	\{(F,K) \in \Sigma_{S_Y} \cap \Gamma:~ \Disc_Y(FK) < X\}.
	\end{align*}
	Then we have 
	\begin{align}\label{transfer}
	N_Y'(X, S_d\times A)
	& = \sum_{\Sigma_{S_Y} \in E_{S_Y}} \# \{(F,K) \in \Sigma_{S_Y} \cap \Gamma:~ \Disc_Y(FK) < X\} \\
	& = \sum_{\Sigma_{S_Y} \in E_{S_Y}} \# \{(F,K) \in \Sigma_{S_Y} \cap \Gamma:~ |\Disc(F)|^{|A|}|\Disc(K)|^d < d_{\Sigma_{S_Y}}X\}, \notag
	\end{align}
	where the quantity $d_{\Sigma_{S_Y}}$ is defined by
	\begin{equation}\label{def:dSigma}
	d_{\Sigma_{S_Y}} := \prod_{p \in S_Y} p^{\Delta_p(F, K)}
	\end{equation} 
	as in \eqref{eq:def_Delta}, and which by Lemma \ref{lem:Delta} is the same for all $(F, K) \in \Sigma_{S_Y} \cap \Gamma$.
	
	As $X \rightarrow \infty$, it is known that
	\begin{align*}
	\#\{F \in \Sigma_{S_Y, d} \cap \mathcal{F}_{S_d}: \ |\Disc(F)| < X\} \sim  C_{\Sigma_{S_Y,d}} X,
	\end{align*}
	\begin{align*}
	\#\{K \in \Sigma_{S_Y, A} \cap \mathcal{F}_A:~ |\Disc(K)| <X\} \sim  C_{\Sigma_{S_Y,A}} X^{a_A} (\log X)^{b_A}
	\end{align*}
	by \ref{A2} and \eqref{eq:abelian1} respectively. By the Product Lemma  \cite[Lemma 3.2]{W21}, there exists a constant $C_{\Sigma_{S_Y}} > 0$ such that
	\begin{align*}
	\# \{(F,K) \in \Sigma_{S_Y} \cap \Gamma:~ |\Disc(F)|^{|A|}|\Disc(K)|^d <  d_{\Sigma_{S_Y}} X\}
	\sim C_{\Sigma_{S_Y}} X^{1/|A|}
	\end{align*}
	as $X \rightarrow \infty$. Therefore by (\ref{transfer}) we may write
	\begin{align}\label{truncate}
	C_Y :=\sum_{\Sigma_{S_Y}}C_{\Sigma_{S_Y}} =\lim_{X \rightarrow \infty} \frac{N_Y'(X, S_d\times A)}{X^{1/|A|}},
	\end{align}
	with the limit being well defined and positive.
	\begin{remark}\label{rk:splitting} 
		We may impose splitting conditions on the $S_d \times A$-fields being counted, and obtain analogues of the same results, as follows.
		
		If in \eqref{eq:def2db} we impose an arbitrary finite set of splitting conditions on $F$ and/or $K$, 
		then we choose $Y$ to be larger than any of the finite primes
		and limit $E_{S_Y}$ to those pairs $(\Sigma_{S, d}, \Sigma_{S,A})$ which are compatible
		with the splitting conditions. We again obtain \eqref{truncate} with the same proof,
		with a smaller value of $C_Y$. 
		
		The rest of the proof then proceeds without change; the splitting conditions
		may be simply dropped in bounds like \eqref{eq:lim0} and		\eqref{eq:discrange}.
		We therefore obtain a version of Theorem \ref{thm:main_general} where finitely many splitting conditions
		may be imposed upon $F$ and/or $K$. Since the splitting conditions on $FK$ are determined by
		those on $F$ and $K$, we may alternatively impose finitely many splitting conditions on $FK$ directly.
	\end{remark}
\subsection{A zero density argument}\label{sec:zero_density}
In this section we prove that the functions $N_Y(X, S_d\times A)$ and
$N_Y'(X, S_d\times A)$, introduced  in \eqref{eq:def2da_2} and \eqref{def:tildeNY} respectively, have the
same asymptotic distribution.
\begin{lemma}
	For arbitrary $Y$, we have
	$$ \lim_{X\to \infty} \frac{N_Y(X, S_d\times A)}{X^{1/|A|}} = C_Y.$$    
\end{lemma}
\begin{proof}
	It suffices to prove that 
	$$\lim_{X\to \infty}\frac{\Delta(X)}{X^{1/|A|}} = 0,$$
	where 
	\begin{equation}
	\begin{aligned}
	\Delta(X) &:= N_Y'(X, S_d\times A)-N_Y(X, S_d\times A) \\
	&= \# \{ (F, K)\in \Gamma \mid \Gal(F^cK/\Q)\neq S_n\times A,  \Disc_Y(FK) \le X \}.
	\end{aligned}
	\end{equation}

	Suppose that $\Gal(F^c K/\Q) \neq S_n \times A$. Then $F^c$ and $K$ must contain a nontrivial common subfield $E$, for which
		$\Gal(F^c/E)$ is normal in $S_n$ with abelian quotient. Therefore, we have $\Gal(F^c/E) = A_n$, so that $E$ is quadratic. 
		Further, for each prime $p$ with $p \mid \mid \Disc(F)$, we have $I_p \not \subset A_n$, which implies that $p$ is ramified in $E$ and hence $K$.
		We thus have the implication $p \mid \mid \Disc(F) \implies p \mid \Disc(K)$.

	For each $Z > Y$, we define 
	\begin{equation}
	\begin{aligned}
	\Delta_Z(X):=\# \{ (F, K)\in \Gamma \mid  \big(p \mid \mid \Disc(F) \implies p \mid \Disc(K)\big) \ \forall p\in (Y, Z), \ \Disc_Y(FK) \le X \},
	\end{aligned}
	\end{equation}
	and $\Delta(X) \leq \Delta_Z(X)$ holds for arbitrary $Z>Y$. Therefore it suffices to prove 
	\begin{equation}\label{eq:lim0}
	\begin{aligned}
	\lim_{Z\to \infty} \lim_{X\to \infty} \frac{\Delta_Z(X)}{X^{1/|A|}} = 0.
	\end{aligned}
	\end{equation}
	
	To prove this, as in \eqref{transfer}, we decompose the set counted by $\Delta_Z(X)$ into a union of subsets,
	for each of which a splitting condition is imposed on $F$ and $K$ at every prime $p \in (Y, Z)$. We use the product lemma and sum the results, in exactly the same way as in \eqref{truncate},
	obtaining the formula
	\begin{equation}\label{eq:deltaz}
	\Delta_Z(X) \sim C \delta_{(Y, Z)} \cdot X^{1/|A|},
	\end{equation}
	where $\delta$ is the local density of those pairs $(F, K)$ satisfying the hypothesis 
	$p \mid \mid \Disc(F) \implies p \mid \Disc(K)$ for all $p \in (Y, Z)$.

    For each subset $S\subset \{ p: Y<p<Z \}$, we consider the contribution to \eqref{eq:deltaz} of those pairs $(F, K)$ 
    where $p|| \Disc(F)$ at $p\in S$ and $p\nparallel \Disc(F)$ at $p\notin S$. 
    Then we have
    $$\delta_{(Y, Z)} \le \sum_{ S} \alpha_S \beta_{S^c},$$
    where \eqref{eq:quant} and \eqref{eq:abelian3} imply that
    $$\alpha_S  \ll C^{|S|}\prod_{p\in S} p^{-1} \cdot \big(p^{-1} + O(p^{-2})\big) $$
    and
    $$\beta_{S^c} \ll \prod_{p \in S^c} \big(1-p^{-1} + O(p^{-2})\big).$$
    Therefore we have
    $$\delta_{(Y, Z)} \ll \sum_{S}  \prod_{p\in S} \big(C p^{-2} + O(p^{-3})\big) \prod_{p\notin S} (1- p^{-1}+O(p^{-2})) = \prod_{p\in (Y, Z)} (1-p^{-1}+O(p^{-2})),$$
    and this quantity tends to $0$ as $Z\to \infty$. 

\end{proof}

\subsection{The tail estimate}
We will need the following crucial bound.
	
	\begin{proposition}\label{crucialprop} 
		We have 
		\begin{align*}
		|N(X, S_d \times A) - N_Y(X, S_d \times A)| \leq o_Y(1) \cdot X^{1/|A|}.
		\end{align*}
	\end{proposition}
\begin{proof}  
	
		We begin by decomposing the set $\Gamma$ of \eqref{eq:def2db} into a disjoint union, tracking the ramification behavior of $(F, K)$ 
		at those primes $p$ where both $F$ and $K$ are ramified.
		For each $\Sigma_S$, we define the set
		\begin{align*}
		\Gamma_{\Sigma_{S}}&:=\{(F,K) \in \Sigma_{S}:~\textrm{$F$ and $K$ are not both ramified at any prime $p \notin S$}\}.
		\end{align*}
		
		We immediately
			obtain the following:
		
		\begin{lemma}\label{disjointunion} We have a disjoint union
			\begin{align*}
			\Gamma = \bigsqcup_{S} \bigsqcup_{\Sigma_S \in \mathcal{R}_S} \Gamma_{\Sigma_S}
			\end{align*}
			where $S$ ranges over all subsets of the rational primes which contain the divisors of $|A|d!$, and for each $S$ the set $\mathcal{R}_S:=\{\Sigma_S\}$ consists of all pairs
			\begin{align*}
			\Sigma_{S} := (\Sigma_{S,d}, \Sigma_{S,A})
			\end{align*}
			such that for each $p \in S$ not dividing $|A|d!$, the \'etale $\Q_p$-algebras determined by $\Sigma_{S,d}$ and $\Sigma_{S,A}$ are ramified over $\Q_p$.
		\end{lemma}

		By Lemma \ref{disjointunion} and Lemma \ref{lem:Delta}, we get 
		\begin{align}
		N(X, S_d \times A) - N_Y(X, S_d \times A) 
		& \leq \#\{(F,K) \in \Gamma:~|\Disc(FK)| < X < \Disc_Y(FK)\} \notag \\
		& = \sum_{S} \sum_{\Sigma_S \in \mathcal{R}_S} 
		\# \{(F,K) \in \Gamma_{\Sigma_S}:~|\Disc(FK)| < X < \Disc_Y(FK)\} \label{eq:discrange} \\
		& \leq \sum_{S} \sum_{\Sigma_S \in \mathcal{R}_S} \# \{(F,K) \in \Gamma_{\Sigma_S}:~|\Disc(FK)| < X \} \label{eq:discrange2} \\
		& \leq \sum_{S} \sum_{\Sigma_S \in \mathcal{R}_S}
		\# \{(F,K) \in \Gamma_{\Sigma_S}:~ |\Disc(F)|^{|A|}|\Disc(K)|^d < d_{\Sigma_S} X \}, \label{differencedecomp}
		\end{align}
		where the quantity $d_{\Sigma_{S}}$ is defined by
		\begin{equation}\label{def:dSigmaS}
		d_{\Sigma_S} := \Delta_{\textnormal{wild}} \prod_{\substack{p \in S \\ p \nmid |A|d!}} p^{\Delta_p(F, K)}
		\end{equation}
		analogously to \eqref{def:dSigma}, and where $\Delta_{\textnormal{wild}}$, the maximum possible product of $p^{\Delta_p(F, K)}$
			over all primes dividing $|A|d!$, may be taken to be a constant.
		Moreover, the sum in \eqref{differencedecomp} is restricted to those $\Sigma_S$ for which $S$ contains at least one prime $p > Y$,
		because any summand in \eqref{eq:discrange} for which $S \subset S_Y$ is zero. 
	
		We now introduce a decomposition of the sum in \eqref{differencedecomp}. 
		
		For each $\Sigma_S = (\Sigma_{S, d}, \Sigma_{S,A})$, and for each $p \in S$ not dividing $|A|d!$, 
		let $g_{\Sigma_{S},d, p} \in S_d$ be a generator of the inertia group $I_{F,p}$, for any $F \in \Sigma_{S,d}$. We define $g_{\Sigma_S, A, p}$ analogously, and write
		$g_{\Sigma_S, p} = (g_{\Sigma_S, d, p}, g_{\Sigma_S, A, p}) \in S_d \times A$. Then the conjugacy classes and indices of these
		elements are well defined; see Section \ref{sec:ant}.
		
		We then let $g_1, \dots, g_m$ be representatives of the nontrivial conjugacy classes in $S_d \times A$,
		and write each $g_k = (g_{k, d}, g_{k, A})$ with $g_{k, d} \in S_d$ and $g_{k, A} \in A$.
		%
		For each squarefree integer $q \in \Z^{+}$ coprime to $|A|d!$, let $S(q)$ denote the set of primes dividing $q$, and define
		\begin{align*}
		V_q:=\{ \vecq = (q_1, \ldots, q_m) \in \Z_{+}^m:~(q_k,q_j)=1, ~ \textrm{$q_k$ squarefree}, ~ \prod_{k=1}^{m}q_k =q\}.
		\end{align*}
		
		For each vector $\vecq \in V_q$, we associate the set of 
		``tame splitting types'' 
		\begin{align*}
		&\Sigma_{\vecq}:=\{\Sigma_{S(q)} \in \mathcal{R}_{S(q)}: 
		~ [g_{\Sigma_{S(q)},p}]=[g_k] ~ \textrm{for each $p|q_k$, $k=1, \ldots, m$}\}
		\end{align*}
		and we say that $(F,K) \in \Sigma_{\vecq}$ if $(F,K) \in \Sigma_{S(q)}$ for some $\Sigma_{S(q)} \in \Sigma_{\vecq}$.
		We then define 
		\begin{align*}
		\Gamma_{\Sigma_{\vecq}}:=
		\{(F,K) \in \Sigma_{\vecq}:~\textrm{$F$ and $K$ are not both  
			ramified at any prime $p \nmid |A|d!q$}\}.
		\end{align*}
		Then, we see that
		\begin{align}\label{lastdecomp}
		\bigsqcup_{S} \bigsqcup_{\Sigma_S \in \mathcal{R}_S} \Gamma_{\Sigma_S} \subset 
		\bigsqcup_{\substack{\textrm{$q$ squarefree} \\ (q, |A|d!) = 1}}\bigsqcup_{(q_1, \ldots, q_m) \in V_{q}}\Gamma_{\Sigma_{\vecq}}
		\end{align}
		by associating to each $(F,K) \in \Gamma_{\Sigma_S}$ 
		the vector $(q_1, \ldots, q_m) \in V_q$, with $q = q(S) := \prod_{p \in S} p$, whose components are defined by
		\begin{align*}
		q_k:=\prod_{\substack{p \in S \\ [g_{\Sigma_{S},p}]=[g_k]}}p, \quad k=1, \ldots, m.
		\end{align*}

		From (\ref{differencedecomp}) and (\ref{lastdecomp}) we get 
		\begin{align*}
		N(X, S_d \times A) - N_Y(X, S_d \times A)
		\leq  \sum_{\substack{q > Y \\ \textrm{$q$ squarefree} \\ (q, |A|d!) = 1}}\sum_{\vecq \in V_q}
		\# \{(F,K) \in \Gamma_{\Sigma_{\vecq}}:~|\Disc(F)|^{|A|} |\Disc(K)|^d < d_{\Sigma_{\vecq}}X\},
		\end{align*} 
		where $d_{\Sigma_{\vecq}}$ is the common value of $d_{\Sigma_S}$ for all $\Sigma_{S(q)} \in \Sigma_{\vecq}$.
				We introduce the dyadic decomposition 
		\begin{align*}
		\R_{\geq 1}^m = \bigsqcup_{\veci = (i_1, \ldots, i_m) \in \Z_{\geq 0}^m}
		\prod_{k=1}^{m}[Q_{i_k}, 2 Q_{i_k}),
		\end{align*}
		with $Q_{i_k}:=2^{i_k}$ for $k=1, \ldots, m$, and write
		\[
		V_{\veci} := \{ \vecq = (q_1, \cdots, q_m): q_k \in [Q_{i_k}, 2 Q_{i_k}) \textnormal{ for each } k \},
		\]
		where the $q_k$ are again assumed coprime to each other and to $|A|d!$.
		Then
		\begin{align*}
		& N(X, S_d \times A) - N_Y(X, S_d \times A) \notag \\ &
		\leq \sum_{\substack{\veci = (i_1, \ldots, i_m) \in \Z_{\geq 0}^m \\ 2^{(i_1 + \cdots + i_m)} > Y/2^{m}}}
		\sum_{\vecq \in V_{\veci}}
		\# \{(F,K) \in \Gamma_{\Sigma_{\vecq}}:~|\Disc(F)|^{|A|} |\Disc(K)|^d < d_{\Sigma_{\vecq}}X\}.
		\end{align*}
		By Lemma \ref{lem:Delta} we have
		\begin{equation}\label{eq:dsigma1}
		d_{\Sigma_{\vecq}} \asymp \prod_{k=1}^{m}Q_{i_k}^{\Delta(g_k)}
		\end{equation}
		for each $\vecq \in V_{\veci}$, where the $\Delta(g, k)$ is the common value of $\Delta_p(F, K)$ for all
		$(F, K) \in \Gamma_{\Sigma_{\vecq}}$ and $p \mid q_k$, given explicitly by
		\begin{equation}\label{eq:dsigma2}
		\Delta(g_k) =  |A| \cdot \ind(g_{k, d}) + d \cdot \ind(g_{k, A}) - \ind(g_k).
		\end{equation}
		We thus have that 
		\begin{align}\label{differencedecomp2}
		& N(X, S_d \times A) - N_Y(X, S_d \times A) \notag \\
		& \leq \sum_{\substack{\veci = (i_1, \ldots, i_m) \in \Z_{\geq 0}^m \\ 2^{(i_1 + \cdots + i_m)} > Y/2^{m}}}
		\sum_{\vecq \in V_{\veci}}
		\# \{(F,K) \in \Gamma_{\Sigma_{\vecq}}:~|\Disc(F)|^{|A|} |\Disc(K)|^d < C \prod_{k=1}^{m}Q_{i_k}^{\Delta(g_k)} \cdot X\},
		\end{align}
		for a constant $C$ depending only on $d$ and $|A|$.
		Each $K$ being counted satisfies $\prod_k q_k^{\ind(g_{k, A})} \mid \Disc(K)$.
		Therefore, for each fixed $F$ and $\vecq$, 
		Theorem \ref{thm:abelian} implies that
		the number of $K$ contributing to \eqref{differencedecomp2} 
		is
		\begin{equation}\label{eq:before_epsilon}
		\ll X^{\frac{a_A}{d}}|\Disc(F)|^{-\frac{|A|a_A}{d}} 
		\prod_{k=1}^{m}Q_{i_k}^{a_A \theta_k}
		\cdot \log \left( \left(  \frac{X \cdot \prod_k Q_{i_k}^{\Delta(g_k)} }{|\Disc(F)|^{|A|}} \right)^{1/d} \right)^{b_A},
		\end{equation}
		with
		\begin{equation}\label{eq:def_theta}
		\theta_k := \frac{\Delta(g_k)}{d}-\ind(g_{k, A}).
		\end{equation}
		
		It will be convenient to eliminate the logarithmic term from \eqref{eq:before_epsilon}:
		given arbitrary
		$\epsilon > 0$, we may choose $\delta(A) > a_A$ with $\delta(A) - a_A$ 
		small,  and so that the expression in \eqref{eq:before_epsilon} is
		
		\begin{equation}\label{eq:epsilon}
		\ll X^{\frac{\delta(A)}{d}}|\Disc(F)|^{-\frac{|A|\delta(A)}{d}} 
		\prod_{k=1}^{m}Q_{i_k}^{\delta(A) \theta_k + \varepsilon}.
		\end{equation}

		Therefore, for each $B$, the contribution to \eqref{differencedecomp2} from those $F$ with $|\Disc(F)| \in [B, 2B)$ is  
		\begin{align}\label{differencedecomp3}
		\ll X^{\delta(A)/d} \sum_{\substack{\veci = (i_1, \ldots, i_m) \in \Z_{\geq 0}^m \\ 2^{(i_1 + \cdots + i_m)} > Y/2^{m}}} \nonumber
		& B^{-\frac{|A|\delta(A)}{d}} \prod_{k=1}^{m}Q_{i_k}^{\delta(A) \theta_k + \varepsilon}
		\times \\ &
		\sum_{\vecq \in V_{\veci}} \# \{F \ : (F, \ast) \in \Gamma_{\Sigma_{\vecq}} :~ B \leq |\Disc(F)| < 2B \}. 
		\end{align}
		
		Here we write $(F, \ast) \in \Gamma_{\Sigma_{\vecq}}$ to indicate that the splitting type $\Gamma_{\Sigma_{\vecq}}$ is imposed on $F$ only,
		i.e. that  $I_p = \langle g_{k,d} \rangle$ for each $p| q_k$.
		
		The averaged uniformity hypothesis (\ref{A3}) provides quantities $r_{g_k}$ (associated to the $g_{k, d}$ component of $g_k$ alone, and 
		satisfying an inequality to be recalled shortly)
		for which the inner sum is 
		$ \ll B \cdot \prod_{k=1}^{m}Q_{i_k}^{r_{g_k}}$, so that the above expression is 
		\begin{align}\label{differencedecomp4}
		\ll X^{\delta(A)/d} \sum_{\substack{\veci = (i_1, \ldots, i_m) \in \Z_{\geq 0}^m \\ 2^{(i_1 + \cdots + i_m)} > Y/2^{m}}}
		B^{1 -\frac{|A|\delta(A)}{d}} \prod_{k=1}^{m}Q_{i_k}^{\delta(A) \theta_k +  r_{g_k} + \varepsilon}.
		\end{align}
		We now sum over dyadic intervals $[B, 2B)$. Since $1 - \frac{|A| \delta(A)}{d} > 0$ (provided that $\delta(A) - a_A$ is chosen to be small), 
		the expression in
		\eqref{differencedecomp4} is bounded, up to an implied constant, by the contribution of the largest possible $B$.
		Since we have $|\Disc(K)| \geq \prod_{k=1}^{m}Q_{i_k}^{\ind(g_{k, A})}$ for each $K$ contributing to \eqref{differencedecomp2}, we get
		\begin{align*}
		|\Disc(F)| \ll X^{1/|A|}\prod_{k=1}^{m}Q_{i_k}^{\frac{d}{|A|} \theta_k}
		\end{align*}
		and we may take this as an upper bound for $B$. Therefore, the expression in \eqref{differencedecomp2} becomes
		\begin{align*}
		& \ll  X^{\delta(A)/d} \sum_{\substack{\veci = (i_1, \ldots, i_m) \in \Z_{\geq 0}^m \\ 2^{(i_1 + \cdots + i_m)} > Y/2^{m}}} 
		\left(X^{1/|A|}\prod_{k=1}^{m}Q_{i_k}^{\frac{d}{|A|} \theta_k} \right)^{1 -\frac{|A|\delta(A)}{d}}
		\prod_{k = 1}^m Q_{i_k}^{\delta(A) \theta_k +  r_{g_k} + \varepsilon} \\
		& \ll X^{1/|A|} 
		\sum_{\substack{\veci = (i_1, \ldots, i_m) \in \Z_{\geq 0}^m \\ 2^{(i_1 + \cdots + i_m)} > Y/2^{m}}}  
		\prod_{k = 1}^m Q_{i_k}^{\frac{d}{|A|}\theta_k +
			r_{g_k} + \varepsilon}.
		\end{align*}
		
	By \eqref{eq:index_bound} in our averaged uniformity hypothesis (\ref{A3}), along with the identity in \eqref{eq:inequ}, we have   
		\begin{equation}\label{eq:beta}
		\beta:=\max_{1 \leq k \leq m}\bigg\{\frac{d}{|A|}\theta_k + r_{g_k} \bigg\} = 
		\max_{1 \leq k \leq m}\bigg\{\frac{\Delta(g_k)}{|A|} - \frac{d \cdot \ind(g_{k, A})}{|A|}  +
		r_{g_k}  \bigg\} < 0.
		\end{equation}
     	        We obtain that  
		\begin{equation}
		N(X, S_d \times A) - N_Y(X, S_d \times A)
		\ll 
		X^{1/|A|} 
		\sum_{\substack{\veci = (i_1, \ldots, i_m) \in \Z_{\geq 0}^m \\ 2^{(i_1 + \cdots + i_m)} > Y/2^{m}}}  
		\prod_{k = 1}^m Q_{i_k}^{\beta + \varepsilon}.
		\end{equation}
		
		For $r \in \Z_+$, we have  
		\begin{align*}
		\# \{(i_1, \ldots, i_m) \in \Z_{\geq 0}^m:~2^{(i_1 + \cdots + i_m)}=2^r\} = 
		{{r + m - 1}\choose{m-1}} .
		\end{align*}
		Then since $\beta < 0$, we get 
		\begin{align*}
		N(X, S_d \times A) - N_Y(X, S_d \times A) & \ll
		\sum_{r=0}^{\infty}\sum_{\substack{(i_1, \ldots, i_m) \in \Z_{\geq 0}^m\\ Y/2^{m} < 2^{(i_1 + \cdots + i_m)}=2^r}}(2^{(i_1 + \cdots + i_m)})^{\beta + \varepsilon}\\
		& \leq \sum_{r=\lceil \frac{\log(Y)}{\log(2)}-m \rceil}^{\infty}{{r + m - 1}\choose{m-1}}(2^{\beta +\varepsilon})^{r} \\
		& \ll \sum_{r=\lceil \frac{\log(Y)}{\log(2)}-m \rceil}^{\infty}r^{m-1}(2^{\beta + \varepsilon})^{r}\\
		& \ll (\log Y)^{m-1} Y^{\beta + \varepsilon},
		\end{align*}
		which completes the proof of Proposition \ref{crucialprop}.

\end{proof}

\subsection{Completion of the proof}
	We now complete the proof of Theorem \ref{thm:main_general}. Recall that 
	\begin{align*}
	C_Y=\lim_{X \rightarrow \infty} \frac{N_Y(X, S_d \times A)}{X^{1/|A|}}.
	\end{align*}
	By \eqref{eq:NYY_ineq}, we have 
	\begin{align*}
	N_{Y_1}(X,S_d \times A) \leq N_{Y_2}(X,S_d \times A)
	\end{align*}
	for $0 < Y_1 < Y_2$, hence the sequence $(C_Y)_{Y > 0}$ is monotone increasing. Now, by Proposition \ref{crucialprop} we have 
	\begin{align*}
	\frac{N(X, S_d \times A)}{X^{1/|A|}} \leq \frac{N_Y(X, S_d \times A)}{X^{1/|A|}} + o_Y(1).
	\end{align*}
	Hence 
	\begin{equation}\label{eq:limsup1}
	\limsup_{X \rightarrow \infty} \frac{N(X, S_d \times A)}{X^{1/|A|}} \leq \limsup_{X \rightarrow \infty} 
	\frac{N_Y(X, S_d \times A)}{X^{1/|A|}}  
	+ o_Y(1) = C_Y + o_Y(1).
	\end{equation}
	By choosing (for example) $Y=1$, we see that the left side of \eqref{eq:limsup1} is bounded. Since by \eqref{eq:NY_ineq} we have 
	\begin{align*}
	\frac{N_Y(X, S_d \times A)}{X^{1/|A|}} \leq \frac{N(X, S_d \times A)}{X^{1/|A|}},
	\end{align*} we see that the sequence $(C_Y)_{Y>0}$ is bounded as well.
	We have shown that $(C_Y)_{Y>0}$ is a bounded, monotone sequence, hence the limit
	\begin{align*}
	C:=\lim_{Y \rightarrow \infty} C_Y
	\end{align*}
	exists. 
	
	Now, again by \eqref{eq:NY_ineq} we have
	\begin{align*}
	C_Y = \liminf_{X\to\infty}{\frac{N_Y(X, S_d \times A)}{X^{1/|A|}}} \leq  
	\liminf_{X\to\infty}{\frac{N(X, S_d \times A)}{X^{1/|A|}}}, 
	\end{align*}
	so that
	\begin{align}\label{Cbound1}
	C = \lim_{Y\to\infty} C_Y \leq \liminf_{X\to\infty}{\frac{N(X, S_d \times A)}{X^{1/|A|}}}. 
	\end{align}
	Again by Proposition \ref{crucialprop} we have
	\begin{align*}
	\limsup_{X \to \infty} \frac{N(X, S_d \times A)}{X^{1/|A|}} & \leq 
	\limsup_{X \to \infty} \frac{N(X, S_d \times A) - N_Y(X, S_d \times A)}{X^{1/|A|}} + 
	\limsup_{X \to \infty} \frac{N_Y(X, S_d \times A)}{X^{1/|A|}}\\
	& \leq o_Y(1) + C_Y.
	\end{align*}
	Let $Y \rightarrow \infty$ to get 
	\begin{align}\label{Cbound2}
	\limsup_{X \to \infty} \frac{N(X, S_d \times A)}{X^{1/|A|}} \leq C.
	\end{align}
	Combining (\ref{Cbound1}) and (\ref{Cbound2}), we conclude that 
	\begin{align*}
	\lim_{X \to \infty} \frac{N(X, S_d \times A)}{X^{1/|A|}} =C.
	\end{align*}
	This completes the proof of Theorem \ref{thm:main_general}. \qed

\section*{Acknowledgments}
We would like to thank Solly Parenti for some preliminary discussions which helped to inspire this project, and also for helpful comments on the completed version of
this paper. We would also like to thank Brandon Alberts, Alexander Slamen, and Melanie Matchett Wood for helpful comments.

RM and WLT were partially supported by grants from the Simons Foundation (No. 421991) and the
National Science Foundation (No. DMS-1757872).
FT was partially supported by grants from the Simons Foundation (Nos. 563234 and 586594).
JW was partially supported by the Foerster-Bernstein Fellowship at Duke University.

Claude was used to proofread and to search for any missed references; AI tools were not used in any more substantial way.

\end{document}